\documentclass[11pt]{amsart}

\usepackage{fullpage}
\usepackage{hyperref}
\usepackage{placeins}

\usepackage{amssymb, amsmath, amsthm, amsfonts, mathtools}
\usepackage{mathabx}

\usepackage{fancyhdr}

\allowdisplaybreaks

\usepackage{hyperref}
\usepackage[nameinlink]{cleveref}

\usepackage{autonum} 

\usepackage[dvipsnames]{xcolor}

\usepackage{graphicx}

\usepackage{comment}

\usepackage{thmtools}

\usepackage{pdfpages}

\usepackage{leftindex}

\theoremstyle{plain}
\newtheorem{theorem}{Theorem}[section]
\crefname{theorem}{Theorem}{Theorems}
\newtheorem{proposition}[theorem]{Proposition}
\crefname{proposition}{Proposition}{Propositions}
\newtheorem{lemma}[theorem]{Lemma}
\crefname{lemma}{Lemma}{Lemmas}
\newtheorem{corollary}[theorem]{Corollary}
\crefname{corollary}{Corollary}{Corollaries}
\theoremstyle{definition}
\newtheorem{definition}[theorem]{Definition}
\crefname{definition}{Definition}{Definitions}
\newtheorem{example}[theorem]{Example}
\crefname{example}{Example}{Examples}
\newtheorem{def-prop}[theorem]{Definition-Proposition}
\crefname{def-prop}{Definition-Proposition}{}
\newtheorem{def-lem}[theorem]{Definition-Lemma}
\crefname{def-lem}{Definition-Lemma}{}

\crefname{assumption}{Assumption}{Assumptions}
\theoremstyle{remark}
\newtheorem{remark}[theorem]{Remark}
\crefname{remark}{Remark}{Remarks}

\crefname{claim}{Claim}{Claims}

\crefname{notation}{Notation}{}

\crefname{question}{Question}{Questions}

\crefname{todo}{To do}{To dos}

\crefname{problem}{Problem}{Problems}

\crefname{chapter}{Chapter}{Chapters}
\crefname{section}{Section}{Sections}
\crefname{subsection}{Subsection}{Subsections}
\crefname{figure}{Figure}{Figures}

\usepackage[figurename=Figure,tablename=Table]{caption}

\newcommand\supp{\mathop{\rm supp}}

\DeclareMathOperator{\spn}{span}

\newcommand{\cl}[1]{\mathcal{#1}}
\newcommand{\bb}[1]{\mathbb{#1}}

\newcommand{\ca}{\mathrm{C}^*}
\newcommand{\ol}{\overline}

\usepackage{graphicx} % Required for inserting images

\title[Groupoid operator systems]{Operator systems in \'etale groupoid C$^*$-algebras}

\author[Baziotis et al]{}
\author[]{Georgios Baziotis}
\address{Research Institute of Mathematics and NextQuantum Innovation Research Center, Department of Physics and Astronomy\\ Seoul National University, Seoul 08826, South Korea}
\email{baziotis@snu.ac.kr}

\author[]{Douglas Farenick}
\address{Department of Mathematics \& Statistics, University of Regina, Regina, Saskatchewan S4S 0A2, Canada}
\email{douglas.farenick@uregina.ca}

\author[]{Evgenios T.A. Kakariadis}
\address{Department of Mathematics, National and Kapodistrian University of Athens, Athens, 1578 84, Greece}
\email{evkakariadis@math.uoa.gr}

\author[]{Malte Leimbach}
\address{Christian-Albrechts-Universit\"at zu Kiel, Mathematisches Seminar, Heinrich-Hecht-Platz 6, 24118 Kiel, Germany.}
\email{leimbach@math.uni-kiel.de}

\author[]{Lyudmila Turowska}
\address{Department of Mathematical Sciences, Chalmers University of Technology and University of Gothenburg, Chalmers Tv\"argata 3, SE-412 96, Gothenburg, Sweden}
\email{turowska@chalmers.se}

\author[]{Walter D. van Suijlekom}
\address{Institute for Mathematics, Astrophysics and Particle Physics, Radboud	University Nijmegen, Heyendaalseweg 135, 65254 AJ Nijmegen, The Netherlands.}
\email{waltervs@math.ru.nl}

\date{October 25, 2026}

\thanks{2020 {\it Mathematics Subject Classification:} 46L05, 46L07, 46L08, 47L55}

\thanks{{\it Key words and phrases:} Operator systems, Cantor systems, C*-envelope, hyperrigidity.}

\thanks{\textit{AI usage (AIL2026):} \texttt{AI-0: Text1 Check1 [QWEN3.8, 2026-10]}}

\begin{document}

%%%%%%%%%%%%%%%%%%%%%%%%%%%
\begin{abstract}
We study operator systems arising from positivity domains of \'etale groupoids with compact unit space. 
Extending Kavruk's \emph{enough unitaries} principle we show that such operator systems are hyperrigid in the C$^*$-algebra they generate. 
We focus on operator systems of positivity domains in UHF algebras and, by identifying them with Cantor graph operator systems, we show that they are hyperrigid.
Furthermore, we show that in general, such a characterization is not possible when the \'{e}taleness condition of the underlying groupoid is removed, due to the failure of reduced spectral and compact operator synthesis.
\end{abstract}
%%%%%%%%%%%%%%%%%%%%%%%%%%%

\maketitle

%%%%%%%%%%%%%%%%%%%%%%%%%%%
\section{Introduction}
%%%%%%%%%%%%%%%%%%%%%%%%%%%

In a recent approach to noncommutative geometry, operator systems have been put forward as a way to refine existing C$^*$-algebraic machinery;
they appear both in the setting of spectral truncations \cite{CvS21} and in the setting of limited spatial resolution \cite{CvS22}.
The principle in the latter is to only be able to distinguish points in a metric space $(X,d)$ up to a tolerance of, say $\varepsilon$.
This can be formalized by considering the \emph{tolerance relation} 
\begin{align} 
    \mathcal{R}_\varepsilon := \{(x,y) \mid d(x,y) < \varepsilon\}
\end{align}
as a symmetric subset containing the diagonal of the pair groupoid $X\times X\rightrightarrows X$.
It seems then natural to study the self-adjoint subspace 
\begin{align}
    \mathcal{E}(\mathcal{R}_\varepsilon) := \mathrm{span}\{f \in \mathrm{C}_\mathrm{c}(X \times X) \mid \supp(f) \subseteq \mathcal{R}_\varepsilon\} 
\end{align}
of the groupoid C$^*$-algebra $\mathrm{C}^*(X \times X) \cong \mathcal{K}(\mathrm{L}^2(X))$. 

Following \cite{CvS22} we go one step further in this paper and consider open symmetric subsets of topological groupoids containing the unit space;
in \emph{loc.cit.}\ such subsets were called `bonds', but we prefer to call them \emph{positivity domains}, in line with \cite{KLTvS26}. 
Explicitly we are interested in the C$^*$-envelopes of operator subsystems of the (universal) groupoid C$^*$-algebras generated by positivity domains, as well as in the question when these operator systems are hyperrigid in the sense of Arveson \cite{Arv11Bdry2}.
Apart from the pair groupoid case which was investigated in \cite{CvS22}, 
we restrict our attention to \'etale groupoids with compact unit space.

If $\Gamma$ is a discrete group, a positivity domain $\Omega \subseteq \Gamma$ is a subset such that $\Omega^{-1} = \Omega \ni e$.
If $\Omega$ contains a generating set for $\Gamma$, then the operator system 
\begin{align} 
    \mathcal{E}(\Omega) := \mathrm{span}\{u_\gamma \mid \gamma \in \Omega\}
\end{align}
contains the unitary generators for $\mathrm{C}^*(\Gamma)$, and it follows from \cite[Proposition 5.6]{Kav14} that the C$^*$-envelope of $\mathcal{E}(\Omega)$ coincides with the group C$^*$-algebra $\mathrm{C}^*(\Gamma)$.
We point out that the C$^*$-envelope is used as an invariant in the approach to the extension problem for partially defined positive semi-definite functions in \cite{KLTvS26}, which in part motivates our interest here. Moreover, 
we mention the \emph{Cantor graph isomorphism game} recently introduced in \cite{BCTT25} which, from the point of view of this article can be regarded as a \emph{positivity domain isomorphism game}.

\par In what follows we outline the structure of the paper. 
In \autoref{sec:env_hyperr} we provide a short introduction to operator systems and their C$^*$-envelopes. 
After briefly  discussing the notion of hyperrigidity, we vary on Kavruk's \emph{enough unitaries} principle, and utilize it in the sequel to generalize the discrete group case to \'etale groupoids with compact unit space in \autoref{groupoid_op_sys}. 
We conclude Section 3 with \autoref{lem:bimod}, showing that self-adjoint operator spaces arising from positivity domains are bimodules over the abelian self-adjoint algebra $\mathrm{C}_\mathrm{c}(\mathcal{G}^{(0)})$ of compactly supported continuous functions over the unit space. 
\cref{sec:cantor,sec:sp-synthesis} are partially concerned with providing examples and counterexamples for the converse of \autoref{lem:bimod}. 
In \autoref{sec:cantor}, we focus on operator subsystems of finite dimensional and UHF C*-algebras arising from positivity domains and characterize them as graph and Cantor graph operator systems, respectively, thus providing a converse to \autoref{lem:bimod}. 
We rely on results by Power \cite{Pow93} which hold in the more general setting of (unital) AF C$^*$-algebras, and in fact our results generalize to that setting as well.
Moreover, utilizing our results from \autoref{sec:env_hyperr}, we show that these operator systems are hyperrigid. 
The bimodularity theme reappears in \autoref{sec:sp-synthesis}, where the necessity of \'etaleness of the underlying groupoid is demonstrated. 
In particular, we provide two cases; \autoref{spectral_synth} in the group setting and \autoref{compact_synth} in the transformation groupoid case, where the converse of \autoref{lem:bimod} fails. 
In both instances, this is achieved through connections with abstract harmonic analysis and specifically due to the failure of (reduced) \emph{spectral} and \emph{compact operator synthesis}.

%%%%%%%%%%%%%%%%%%%%%%%%%%%
\section*{Acknowledgments}
The authors would like to thank Lara Ismert and Ivan G. Todorov, for helpful discussions. 
This note is the result of a research session during the BIRS Workshop \lq\lq Operator systems and their applications'' in February 2025, and its continuation during the program on \lq\lq Operator Algebras and Quantum Information'' at Institut Mittag-Leffler in Spring 2026.
The authors are grateful to BIRS for hospitality during the BIRS Workshop, and to the Swedish Research Council for their support under grant no.\ 2021-06594 while the authors were in residence at Institut Mittag-Leffler in Djursholm, Sweden.

GB was supported by the National Research Foundation of Korea (NRF) grant funded by the Korea government (MSIT) (No. RS-2025-00561391, No. RS-2024-00413957 and RS-2022-NR069971) and by NSF grants 2115071 and 2154459. 
DF was supported by the NSERC Discovery Grant RGPIN-2025-05625. 
EK acknowledges the support from the European Union - Next Generation EU (Implementation Body: HFRI. Project name: Noncommutative Analysis: Operator Systems and Nonlocality. HFRI Project Number: 015825).
ML was supported by NWO under grant OCENW.KLEIN.376, and by the Max Planck Institute for Mathematics. 
LT was supported by the Swedish Research Council project grant 2023-04555.

%%%%%%%%%%%%%%%%%%%%%%%%%%%%%%%%
\subsection*{Open access statement}
%%%%%%%%%%%%%%%%%%%%%%%%%%%%%%%%

For open access purposes, the authors have applied a Creative Commons Attribution (CC BY) license to any Author Accepted Manuscript (AAM) version arising.

%%%%%%%%%%%%%%%%%%%%%%%%%%%%%%%%
\subsection*{Conflict of interest}
%%%%%%%%%%%%%%%%%%%%%%%%%%%%%%%%

The authors declare that they have no conflict of interest.

%%%%%%%%%%%%%%%%%%%%%%%%%%%%%%%%
\subsection*{AI usage}
%%%%%%%%%%%%%%%%%%%%%%%%%%%%%%%%

Following the AIL2026 classification~\cite{Prinz2026AIL}, this paper is tagged \texttt{AI-0: Text1 Check1 [QWEN3.8, 2026-10]}. 
The research, including all ideas, results and proofs, was carried out by the authors without AI assistance (\texttt{AI-0}). 
After completion, the manuscript was passed through QWEN for a language check and a sanity check of the mathematics; only a few minor editorial suggestions were adopted (\texttt{Text1}), and the check found no errors that required changes to the mathematics (\texttt{Check1}). 
The authors take full responsibility for the content.

%%%%%%%%%%%%%%%%%%%%%%%%%%%
\section{C$^*$-envelope and hyperrigidity}\label{sec:env_hyperr}
%%%%%%%%%%%%%%%%%%%%%%%%%%%

Abstractly, an operator system is a matrix ordered complex $^*$-vector space $\mathcal{S}$ with a distinguished Archimedean order unit $e_{\mathcal{S}}$. 
The matrix ordering and Archimedean order unit endow $\mathcal{S}$ with a matrix norm structure; however, the definition of operator system typically does not make additional assumptions about whether $\mathcal{S}$ is complete with respect to this norm. 
The embedding theorem of Choi and Effros \cite{choi--effros1977} asserts that every operator system $\mathcal{S}$ admits a unital complete order embedding $\phi:\mathcal{S}\rightarrow B(H)$ (i.e.\ a unital complete order isomorphism onto its image), for some Hilbert space $H$, such that $\phi(e_{\mathcal{S}})$ is the identity operator on $H$. 
Therefore, operator systems may be viewed more concretely as self-adjoint linear subspaces (possibly not norm-closed) of unital C$^*$-algebras that contain the multiplicative identity of the C$^*$-algebra.
In principle, there are many choices for the embedding $\phi$, and distinct embeddings $\phi_1$ and $\phi_2$ may result in non-isomorphic C$^*$-algebras $\mathrm{C}^*(\phi_1(\mathcal{S}))$ and $\mathrm{C}^*(\phi_2(\mathcal{S}))$ generated by the copies $\phi_1(\mathcal{S})$ and $\phi_2(\mathcal{S})$ of $\mathcal{S}$.

Although (unitally completely order) isomorphic copies of an operator system $\mathcal{S}$ may generate C$^*$-algebras with strikingly different properties, there
is a minimal C$^*$-algebra generated by $\mathcal{S}$, called the \emph{C$^*$-envelope of $\mathcal{S}$}, that is co-universal for all C$^*$-algebras generated by isomorphic copies of $\mathcal{S}$:
the minimal C$^*$-algebra is a quotient of any other generated C$^*$-algebra. 
More precisely, given an operator system $\mathcal{S}$, its C$^*$-envelope $\mathrm{C}^*_\mathrm{e}(\mathcal{S})$ is the C$^*$-algebra generated by the image of $\mathcal{S}$ under a unital complete order embedding $\iota : \mathcal{S} \rightarrow \mathcal{B}(H)$, such that the following universal property holds:
for every unital complete order embedding $\phi : \mathcal{S} \rightarrow \mathcal{B}(K)$, there is a $^*$-homomorphism $\pi : \mathrm{C}^*(\phi(\mathcal{S})) \rightarrow \mathrm{C}^*_\mathrm{e}(\mathcal{S})$, such that $\pi \circ \phi = \iota$.
The C$^*$-envelope $\mathrm{C}^*_\mathrm{e}(\mathcal{S})$ is constructed for instance as the C$^*$-algebra generated by the image of $\mathcal{S}$ inside its \emph{injective envelope} \cite{Ham79}, or as the image of a \emph{maximal} unital complete order embedding \cite{DM05}, or as the direct sum of the images of all \emph{boundary representations} \cite{DK15}.
See \cite{CvS22} for the C$^*$-envelope of a (possibly non-unital) self-adjoint space of operators, and \cite{CKKP26} for characterizations of the correct notion of embedding in this category of self-adjoint operator spaces.

Following \cite{Arv10Bdry3}, an operator subsystem $\mathcal{S}$ of a unital C$^*$-algebra $\mathcal{A}$ is said to be \emph{reduced} if $\mathrm{C}^*_\mathrm{e}(\mathcal{S})=\mathrm{C}^*(\mathcal{S})$. 
A natural problem is to determine, whether a given operator system $\mathcal{S}$ is reduced.
This is a recurring theme in the present paper, related to the notion of hyperrigidity.

The \emph{hyperrigidity} property was introduced by Arveson \cite{Arv11Bdry2} for a generating set $\mathcal{R}$ in a C$^*$-algebra $\mathrm{C}^*(\mathcal{R})$. 
Such a set $\mathcal{R}$ will be called \emph{hyperrigid} (in $\mathrm{C}^*(\mathcal{R})$) if for every faithful non-degenerate $^*$-representation $\Phi \colon \mathrm{C}^*(\mathcal{R}) \to \mathcal{B}(H)$, and every sequence of unital completely positive (ucp) maps $\psi_n \colon \mathcal{B}(H) \to \mathcal{B}(H)$, the following holds:
\[
\lim_n \psi_n(\Phi(f)) = \Phi(f), \forall f \in \mathcal{R} \;
\Longrightarrow \;
\lim_n \psi_n(\Phi(a)) = \Phi(a), \forall a \in \mathrm{C}^*(\mathcal{R}),
\]
where the limits are taken in the norm topology.
As proved by Arveson \cite[Theorem 2.1]{Arv11Bdry2}, hyperrigidity for operator systems is connected with the unique extension property for 
%restrictions of unital $^*$-homomorphisms: 
ucp maps: 
namely, a separable operator system $\mathcal{S}$ is \emph{hyperrigid} in $\mathrm{C}^*(\mathcal{S})$ if and only if, for every unital $^*$-representation $\pi:\mathrm{C}^*(\mathcal{S})\rightarrow B(H)$, the induced completely positive linear map $\pi_{\vert\mathcal{S}}:\mathcal{S}\rightarrow B(H)$ has a unique completely positive extension to $\mathrm{C}^*(\mathcal{S})$. 

A \emph{boundary representation} for an operator system $\mathcal{S}$ generating a unital C$^*$-algebra $\mathcal{A}$ is an irreducible unital $^*$-representation $\rho:\mathcal{A}\rightarrow B(H)$ for which $\rho_{\vert\mathcal{S}}:\mathcal{S}\rightarrow B(H)$
has a unique ucp extension to $\mathcal{A}$. 
Therefore, if $\mathcal{S}$ is hyperrigid in $\mathcal{A}$, every irreducible representation $\rho$ of $\mathcal{A}$ is a boundary representation for $\mathcal{S}$, which implies $\mathcal{A}$ is the C$^*$-envelope of $\mathcal{S}$ \cite{Arv11Bdry2,DK15}.

Arveson conjectured that a converse holds \cite[Conjecture 4.3]{Arv11Bdry2}, i.e. that an operator subsystem $\mathcal{S}$ of unital separable C$^*$-algebra $\mathcal{A}$ generated by $\mathcal{S}$ will be hyperrigid in $\mathcal{A}$, if every irreducible representation of $\mathcal{A}$ is a boundary representation for $\mathcal{S}$. 
This was disproven recently \cite{bilich--dor-on2024,bilich--scherer--2026}, hence decoupling hyperrigidity from irreducible representations. 
Among the main contributions of this paper is the identification of a number of classes of operator systems that are hyperrigid in the C$^*$-algebras they generate.

The multiplicative domain of a ucp map will be a useful tool in studying the restrictions of unital $^*$-homomorphisms of unital C$^*$-algebras
to operator subsystems.  
Recall that the \emph{multiplicative domain} of a ucp map $\phi:\mathcal{A}\rightarrow\mathcal{B}$ of unital C$^*$-algebras is the set $\mathcal M_\phi$ consisting of all $a\in\mathcal{A}$ for which $\phi(a^*a)=\phi(a)^*\phi(a)$ and $\phi(aa^*)=\phi(a)\phi(a)^*$. 
Choi showed that $\mathcal M_\phi$ is a unital C$^*$-subalgebra of $\mathcal{A}$ and the restriction of $\phi$ to $\mathcal M_\phi$ is a $^*$-homomorphism 
\cite[Proposition 1.5.7]{BO08}. 
Thus, if $\mathcal{S}\subseteq\mathcal{A}$ is an operator subsystem for which $\mathcal{S}\subseteq\mathcal M_\phi$, then $\phi_{\vert{\mathrm{C}^*(\mathcal{S})}}:\mathrm{C}^*(\mathcal{S})\rightarrow\mathcal{B}$ is a $^*$-homomorphism. In the case of the restriction a unital $^*$-homomorphism
$\pi:\mathcal{A}\rightarrow\mathcal{B}$ to an operator subsystem $\mathcal{S}$ for which $\mathcal{A}=\mathrm{C}^*(\mathcal{S})$, 
if $\phi:\mathcal{A}\rightarrow\mathcal{B}$ is a ucp map such $\phi_{\vert\mathcal{S}}=\pi_{\vert\mathcal{S}}$ and 
$\mathcal{S}\subseteq\mathcal{M}_\phi$, then $\phi=\pi$,
since the hypotheses imply $\phi$ is a $^*$-homomorphism of $\mathcal{A}$ and the $^*$-homomorphisms
$\pi$ and $\phi$ agree on the generating set $\mathcal{S}$ of $\mathcal{A}$.

%%%%%%%%%%%%%%%%%%%%%%%%%%%
\begin{proposition}\label{main_thm}
Let $\cl A \subseteq \cl B(K)$ be a unital C$^*$-algebra and let $\cl X = \mathrm{span}(\mathcal{R}) \subseteq \cl B(K)$, for some subset $\mathcal{R} \subseteq \mathcal{B}(K)$. 
Suppose that
\[
f^*f, ff^* \in \cl A, \text{ for all } f \in \cl R.
\]
Then the operator system 
\[
\cl S := \cl X^* + \cl A + \cl X
\]
is hyperrigid in $\ca(\cl S)$.
In particular we have $\mathrm{C}^*_\mathrm{e}(\mathcal{S}) \cong \mathrm{C}^*(\mathcal{S}) \subseteq \mathcal{B}(K)$.
The same conclusions hold for $\overline{\mathcal{S}}$.
\end{proposition}

\begin{proof}
Let $\rho \colon \ca(\cl S) \to \cl B(H)$ be a unital $^*$-representation and let $\psi \colon \ca(\cl S) \to \cl B(H)$ be a ucp map such that $\psi_{\vert\cl S} = \rho_{\vert\cl S}$.
In order to deduce hyperrigidity, it suffices to show that $\cl S$ is inside the multiplicative domain $\cl M_\psi$ of $\psi$.
As the multiplicative domain is itself a C$^*$-subalgebra of $\ca(\cl S)$ it suffices to show that $\cl A, \cl R \subseteq \cl M_\psi$.

For $a \in \cl A$ we have that $\psi(a) = \rho(a)$ and $\rho $ is a $^*$-representation, hence
\[
\psi(a^*a) = \rho(a^*a) = \rho(a)^* \rho(a) = \psi(a)^* \psi(a);
\]
likewise $\psi(aa^*) = \psi(a) \psi(a)^*$, and thus $\cl A \subseteq \cl M_\psi$.

Next, for $f \in \cl R$ we have by assumption that $f^*f \in \cl A$ and, therefore
\[
\psi(f^* f) = \rho(f^* f) = \rho(f)^* \rho(f) = \psi(f)^* \psi(f);
\]
likewise $\psi(ff^*) = \psi(f)\psi(f)^*$, and thus $\cl R \subseteq \cl M_\psi$ as required.
\end{proof}

An operator subsystem $\mathcal{S}$ of a C$^*$-algebra $\mathcal{A}$ is said to \emph{contain enough unitaries} \cite{Kav14} if $\mathcal{S}$ contains the unitary generators of $\mathcal{A}$.
The statement about the C$^*$-envelope in the following corollary is \cite[Proposition 5.6]{Kav14}.

%%%%%%%%%%%%%%%%%%%%%%%%%%%
\begin{corollary}
Let $\cl R \subseteq \cl B(K)$ be a set of unitaries.
Then the operator system
\[
  \cl S_{\cl R} := \ol{\spn}\{\mathbf{I}_K, u, u^* \mid u \in \cl R\}
\]
is hyperrigid inside $\ca(\cl R) \subseteq \cl B(K)$.
In particular, $\ca(\cl R)$ is the C$^*$-envelope of $\cl S_{\cl R}$.
\end{corollary}

\begin{proof}
Note that $\mathcal{S}_\mathcal{R} = \overline{\mathcal{X}^* + \mathcal{A} + \mathcal{X}}$ with $\mathcal{X} := \mathrm{span}(\mathcal{R})$ and $\mathcal{A} := \mathbb{C}\mathbf{I}_K$. 
The claim then follows from \autoref{main_thm}.
\end{proof}

%%%%%%%%%%%%%%%%%%%%%%%%%%%
\section{Operator systems induced by positivity domains in \'etale groupoids with compact unit space}
%%%%%%%%%%%%%%%%%%%%%%%%%%%

Let $\cl G \rightrightarrows \mathcal{G}^{(0)}$ be a second countable locally compact Hausdorff groupoid with Haar system.
We denote the source and range maps by $s$ and $r$ respectively.
If $x,y \in \mathcal{G}$ are two groupoid elements with $s(y)=r(x)$, we denote their composition by $y \circ x$.
Recall that the unit space $\cl G^{(0)}$ is a closed subset of $\cl G$, \cite[Lemma 2.3.2]{Sim18}. 
The groupoid $\cl G$ is called \emph{\'etale} if the source or the range map $s,r:\cl G\to\cl G^{(0)}$ is a local homeomorphism.
In this case both, $s$ and $r$, are local homeomorphisms and the unit space $\cl G^{(0)}$ is open in $\cl G$, \cite[Lemma 2.4.2]{Sim18}. 

A subset $B$ of an \'etale groupoid $\mathcal{G}$ is a \emph{bisection} if there is an open set $U \subseteq \mathcal{G}$ containing $B$ such that $r:U\to r(U)$ and $s:U\to s(U)$ are both homeomorphisms onto open subsets of $\cl G^{(0)}$.

The set of compactly supported functions on $\mathcal{G}$ is a $^*$-algebra when equip\-ped with the convolution product; in the \'etale case this convolution product is given by
\begin{align}
    (f \ast g)(x)
    := \sum_{y \in \mathcal{G}_{s(x)}} f(x\circ y^{-1}) g(y),
\end{align}
for all functions $f, g \in \mathrm{C}_\mathrm{c}(\mathcal{G})$, and all elements $x \in \mathcal{G}$. 
Here 
\begin{align} 
\mathcal{G}_u := \{y\in\mathcal{G} \mid s(y)=u\}
\end{align}
denotes the set of all elements with source $u\in\mathcal{G}^{(0)}$.
The completion of $\mathrm{C}_\mathrm{c}(\cl G)$ under the I-norm is the universal groupoid C$^*$-algebra denoted by $\mathrm{C}^*(\mathcal{G})$ \cite{Ren80}.
For details we refer to \cite{Sim18} in the \'etale case and to \cite{Wil19} in the general locally compact case.
At this point we only recall the fact that the groupoid C$^*$-algebra is densely spanned by functions compactly supported in bisections \cite[Lemma 3.1.3]{Sim18}, if $\mathcal{G}$ is \'etale:
\begin{align}
    \mathrm{C}^*(\mathcal{G})
    = \overline{\mathrm{C}_\mathrm{c}(\mathcal{G})}
    = \overline{\mathrm{span}}\{f \in \mathrm{C}_\mathrm{c}(\mathcal{G}) \mid \supp(f) \text{ is a bisection}\}.
\end{align}
In what follows we often identify $\mathrm{C}_\mathrm{c}(\mathcal{G})$ with its image in $\mathrm{C}^*(\mathcal{G})$. 
It should be clear from the context whether we regard it as a subspace of the C$^*$-algebra $\mathrm{C}_\mathrm{0}(\mathcal{G})$ or of $\mathrm{C}^*(\mathcal{G})$.

For a subset $\Omega \subseteq \mathcal{G}$ and $n\in\mathbb{N}$ we denote by $\Omega^{(n)}$ the set of $n$-tuples $(y_1,\dots,y_n) \in \Omega^n$ such that $s(y_i)=r(y_{i+1})$, for all $i=1,\dots,n-1$.

%%%%%%%%%%%%%%%%%%%%%%%%%%%
\begin{definition}
Let $\mathcal{G}$ be a locally compact Hausdorff groupoid.
An open subset $\Omega \subseteq \mathcal{G}$ is called a \emph{positivity domain} if $\Omega^{-1} = \Omega \supseteq \mathcal{G}^{(0)}$.

We denote by
\begin{align}
    \mathcal{G}(\Omega) := \bigcup_{n\in\mathbb{N}} \left\{ y_1\circ\cdots\circ y_n \mid (y_1,\dots,y_n) \in \Omega^{(n)} \right\}
        \rightrightarrows \mathcal{G}^{(0)}
\end{align}
the \emph{subgroupoid generated} by the positivity domain $\Omega$.
We say that $\Omega$ is \emph{non-degenerate} if $\overline{\mathcal{G}(\Omega)} = \mathcal{G}$.
\end{definition}

%%%%%%%%%%%%%%%%%%%%%%%%%%%
\begin{definition}
Let $\mathcal{G}$ be a locally compact Hausdorff groupoid with Haar system, and let $\Omega \subseteq \mathcal{G}$ be an open subset.
We write 
\begin{align}
   \mathcal{E}(\Omega) := \mathrm{span}\{f\in\mathrm{C}_\mathrm{c}(\mathcal{G}) \mid \supp(f) \subseteq \Omega\} \subseteq \mathrm{C}^*(\mathcal{G}).
\end{align}
\end{definition}

Note that if $\Omega$ is a positivity domain, then $\mathcal{E}(\Omega) \subseteq \mathrm{C}^*(\mathcal{G})$ is a self-adjoint subspace.
If moreover $\mathcal{G}^{(0)}$ is compact then $\mathcal{E}(\Omega)$ is unital, hence an operator system.
If $\Omega$ is a non-degenerate positivity domain, we have $\mathrm{C}^*(\mathcal{E}(\Omega)) = \mathrm{C}^*(\mathcal{G})$.

%%%%%%%%%%%%%%%%%%%%%%%%%%%
\begin{example}\label{ex:Tolerance-relation}
Let $X$ be a locally compact (Borel) measure space and consider the pair groupoid $\mathcal{G} := X \times X \rightrightarrows X$.
Let $\Omega \subseteq X \times X$ be a positivity domain.
Then $\Omega$ is a \emph{tolerance relation} on the set $X$, and the subgroupoid $\mathcal{G}(\Omega)\subseteq X \times X$ generated by $\Omega$ is an equivalence relation.
Note that $\Omega$ viewed as a positivity domain in the groupoid $\mathcal{G}(\Omega) \rightrightarrows X$ is non-degenerate.
The non-unital operator system $\mathcal{E}(\Omega)$ is studied in detail in \cite{CvS22}; 
in particular it is shown in \cite[Theorem 4.11]{CvS22} that if $X$ is a complete, locally compact path metric measure space with a measure $\mu$ of
full support then for $\Omega = \mathcal{R}_\varepsilon = \{(x,y) \mid d(x,y) < \varepsilon\}$ we have $\mathrm{C}^*_\mathrm{e}(\mathcal{E}(\mathcal{R}_\varepsilon)) = \mathrm{C}^*(X \times X) \cong \mathcal{K}(\mathrm{L}^2(X,\mu))$. 
\end{example}

%%%%%%%%%%%%%%%%%%%%%%%%%%%
\begin{theorem}\label{groupoid_op_sys}
Let $\cl G \rightrightarrows \mathcal{G}^{(0)}$ be an \'etale groupoid and assume that the unit space $\cl G^{(0)}$ is compact.
Let $\Omega \subseteq \cl G$ be a positivity domain.
Then the operator system $\mathcal{E}(\Omega)$ is hyperrigid, and the generated C$^*$-subalgebra $\mathrm{C}^*(\mathcal{E}(\Omega)) \subseteq \mathrm{C}^*(\mathcal{G})$ coincides with the C$^*$-envelope, $\mathrm{C}^*_\mathrm{e}(\mathcal{E}(\Omega))$.
In particular, if the positivity domain $\Omega$ is non-degenerate we have $\mathrm{C}^*_\mathrm{e}(\mathcal{E}(\Omega)) = \mathrm{C}^*(\mathcal{G})$.
\end{theorem}

\begin{proof}
Set 
\begin{align}
  \mathcal{A} &:= \mathrm{span}\{f\in \mathrm{C}_\mathrm{c}(\mathcal{G}) \mid \supp(f) \subseteq \mathcal{G}^{(0)}\}
\end{align}
and
\begin{align}
    \mathcal{X} &:= \mathcal{E}(\Omega) = \mathrm{span}\{f \in \mathrm{C}_\mathrm{c}(\mathcal{G}) \mid \supp(f) \subseteq \Omega\}.
\end{align}
As $\cl G$ is \'etale and $\Omega$ is open by assumption, for any $f\in\mathrm{C}_\mathrm{c}(\Omega)$ we can cover $\supp(f)$ with open bisections, which are subsets of $\Omega$, and then use compactness to pass to a finite subcover $U_1,\ldots, U_n$. 
Choose a partition of unity $\{h_i\}_{i=1}^n$ subordinated to $\{U_i\}_{i=1}^n$.
Then $f=\sum_{i=1}^n fh_i$ and $f_i:=fh_i\in \mathrm{C}_\mathrm{c}(\Omega)$ are supported in $U_i$. 
Therefore 
\begin{align}
    \mathcal{E}(\Omega)=\mathrm{span}\{f\in \mathrm{C}_\mathrm{c}(\Omega) \mid \supp(f) \text{ is a bisection}\}.    
\end{align}
By \cite[Lemma 3.14]{Sim18}, we have that if $f \in \mathrm{C}_\mathrm{c}(\mathcal{G})$ is supported in a bisection, then $f^*\ast f, f\ast f^*\in \mathrm{C}(\cl G^{(0)}) = \mathcal{A}$.
Noting that $\mathcal{E}(\Omega) = \mathcal{X}^* + \mathcal{A} + \mathcal{X}$, we now see from \autoref{main_thm} that $\mathcal{E}(\Omega)$ is hyperrigid and 
\begin{align}\label{eqn:C*-env-C*-Omega}
    \mathrm{C}^*_\mathrm{e}(\mathcal{E}(\Omega)) = \mathrm{C}^*(\mathcal{E}(\Omega)) \subseteq \mathrm{C}^*(\mathcal{G}).
\end{align}
Moreover, if $\Omega$ is non-degenerate we have equality in (\ref{eqn:C*-env-C*-Omega}). 
\end{proof}

We finish this section with the observation that the operator systems $\mathcal{E}(\Omega)$ are bimodular for the action of the abelian 
algebra $\mathrm{C}_\mathrm{c}(\mathcal{G}^{(0)})$.
For an \'etale groupoid $\mathcal{G}$, $\mathcal{G}^{(0)}$ is open and hence $\mathrm{C}_\mathrm{c}(\mathcal{G}^{(0)})$ can be considered as a subalgebra of $\mathrm{C}_\mathrm{c}(\mathcal{G})$ -- by extending the functions on $\mathcal{G}^{(0)}$ by $0$ to all of $\mathcal{G}$;
note that in the \'etale case $\mathrm{C}_0(\mathcal{G}^{(0)})$ is even a C$^*$-subalgebra of $\mathrm{C}^*(\mathcal{G})$. 
In this case one readily computes that, for all, $f \in \mathrm{C}_\mathrm{c}(\mathcal{G})$ and $g \in \mathrm{C}_\mathrm{c}(\mathcal{G}^{(0)}) \subseteq \mathrm{C}_\mathrm{c}(\mathcal{G})$,  we have
\begin{equation}\label{eq:conv}
(f \ast g) (x) = f(x)g(s(x)) \text{ and }(g \ast f)(x) = g(r(x))f(x), \text{ for all }x\in\mathcal{G}.
\end{equation}
For $T \in \mathrm{C}^*(\mathcal{G})$, and $g,h\in\mathrm{C}_\mathrm{c}(\mathcal{G}^{(0)})$, we denote the thus induced left and right actions by $g \cdot T \cdot h \in \mathrm{C}^*(\mathcal{G})$. 
If $\mathcal{G}$ is a general locally compact Hausdorff groupoid with Haar system, we define the left and right $\mathrm{C}_\mathrm{c}(\mathcal{G}^{(0)})$-actions on $\mathrm{C}_\mathrm{c}(\mathcal{G})$ by the right-hand sides of (\ref{eq:conv}) and extend them to actions on $\mathrm{C}^*(\mathcal{G})$. 
We say that a subspace $E \subseteq \mathrm{C}^*(\mathcal{G})$ is \emph{$\mathrm{C}_\mathrm{c}(\mathcal{G}^{(0)})$-bimodular} if 
\begin{align} 
\mathrm{C}_\mathrm{c}(\mathcal{G}^{(0)}) \cdot E \subseteq E
\, \text{ and } \,
E \cdot \mathrm{C}_\mathrm{c}(\mathcal{G}^{(0)}) \subseteq E.
\end{align}

%%%%%%%%%%%%%%%%%%%%%%%%%%%
\begin{lemma}\label{lem:bimod}
Let $\mathcal{G}$ be a locally compact Hausdorff groupoid with Haar system, and let $\Omega \subseteq \mathcal{G}$ be a positivity domain.
Then the self-adjoint operator space $\mathcal{E}(\Omega)$ is $\mathrm{C}_\mathrm{c}(\mathcal{G}^{(0)})$-bimodular.
\end{lemma}

\begin{proof}
Let $f \in \mathrm{C}_\mathrm{c}(\Omega)$, and $g \in \mathrm{C}_\mathrm{c}(\mathcal{G}^{(0)})$. 
It follows  from (\ref{eq:conv}) that
\begin{align} 
\supp(f \ast g) \cup \supp(g \ast f) \subseteq \supp(f) \subseteq \Omega,
\end{align}
and therefore $\mathcal{E}(\Omega) = \mathrm{C}_\mathrm{c}(\Omega)$ is $\mathrm{C}_\mathrm{c}(\mathcal{G}^{(0)})$-bimodular.
\end{proof}
In the next two sections we address the question whether the converse to this lemma holds, i.e.\ given a $\mathrm{C}_\mathrm{c}(\mathcal{G}^{(0)})$-bimodular self-adjoint operator subspace $E$ of $\mathrm{C}^*(\mathcal{G})$, whether there is a positivity domain $\Omega \subseteq \mathcal{G}$ such that $E = \overline{\mathcal{E}(\Omega)}$.

%%%%%%%%%%%%%%%%%%%%%%%%%%%
\section{Operator systems arising from graphs and Cantor graphs}\label{sec:cantor}
%%%%%%%%%%%%%%%%%%%%%%%%%%%

In this section we restrict ourselves to operator systems arising from finite and infinite graphs, and show that they are hyperrigid in the C$^*$-algebra they generate.  
We begin by spelling out \autoref{ex:Tolerance-relation} in the finite case.

%%%%%%%%%%%%%%%%%%%%%%%%%%%
\begin{example}\label{ex:Graph}
In \autoref{ex:Tolerance-relation} take $\mathcal{X}$ to be the set with $n$ elements, $\mathcal{X} := [n]$.
Then positivity domains $\Omega$ in the pair groupoid $\mathcal{G} := [n] \times [n] \rightrightarrows [n]$ correspond exactly to undirected graphs on $n$ vertices (say, without self-loops), and the subgroupoid $\mathcal{G}(\Omega) \subseteq [n] \times [n]$ is the path space of the graph $\Omega$.  
Note that $\mathcal{G}(\Omega) = [n] \times [n]$ if and only if the graph $\Omega$ is connected.

The groupoid C$^*$-algebra of the pair groupoid is the full matrix algebra $\mathrm{C}^*([n]\times[n]) = \mathcal{M}_n$, and the operator subsystem $\mathcal{E}(\Omega)$ is the graph operator system $\mathcal{S}_\Omega$ spanned by the matrix units $E_{ij}$, for $(i,j)$ an edge of $\Omega$, or $i=j$.
The C$^*$-subalgebra generated by $\mathcal{E}(\Omega)$ is
\begin{align}\label{eqn:C*-graph}
    \mathrm{C}^*(\mathcal{E}(\Omega)) = \bigoplus_{l=1}^k \mathcal{M}_{n_l} 
    =\mathrm{C}^*(\mathcal{G}(\Omega)),    
\end{align}
where $k$ is the number of connected components of the graph $\Omega$ and $n_l$ the number of vertices in the $l$-th connected component; 
by \cite[Theorem 3.2]{OP15} we have that the C$^*$-envelope $\mathrm{C}^*_\mathrm{e}(\mathcal{S}_\Omega)$ of the graph operator system coincides with (\ref{eqn:C*-graph}). 
In particular, if the graph $\Omega$ is connected, then $\mathrm{C}^*_\mathrm{e}(\mathcal{S}_\Omega) = \mathcal{M}_n$.
\end{example}

As explained in \autoref{ex:Graph}, the result of Ortiz and Paulsen \cite{OP15} shows the C$^*$-envelope $\mathrm{C}^*_\mathrm{e}(\mathcal{S}_\Omega)$ of $\cl S_\Omega$ is given by $ \mathrm{C}^*(\mathcal{E}(\Omega))$.
\autoref{groupoid_op_sys} yields the following additional information.

%%%%%%%%%%%%%%%%%%%%%%%%%%%
\begin{corollary}\label{pr_gr_hpr}
If $\Omega$ is an undirected graph on $n$ vertices, then the graph operator system $\cl S_\Omega$ is hyperrigid.
\end{corollary}

\par We conclude this section with a discussion of Cantor graphs and the associated operator systems realized in the UHF setting. 
The interplay between graphs and certain closed, open and symmetric subsets of the space $\prod_{n=1}^\infty[d_n]\times \prod^{\infty}_{n=1}[d_n]$ called \emph{Cantor graphs} was initiated in \cite{Pow93}. In the operator system picture, it was shown in \cite{MT18} that these correspond precisely to concrete operator systems in UHF C$^*$-algebras that are bimodules over the canonical maximal abelian self-adjoint algebra, extending the work in \cite{PPS89,OP15}.   

Recall that a \emph{UHF algebra} is a C$^*$-algebra $\mathcal{A}_\infty$ which arises as the inductive limit of an inductive system 
\begin{align} 
\mathcal{M}_{n_1}\stackrel{\pi_{2,1}}{\longrightarrow} 
\cdots
\stackrel{\pi_{k+1,k}}{\longrightarrow}
\mathcal{M}_{n_{k+1}}\stackrel{\pi_{k+2,k+1}}{\longrightarrow}
\cdots,
\end{align}
where $\pi_{k+1,k}:\mathcal{M}_{n_k}\rightarrow \mathcal{M}_{n_{k+1}}$ is a unital $^*$-homomorphism for every $k\in\bb{N}$.
For every $k\in\mathbb{N}$, we denote by $\pi_{\infty,k} : \mathcal{M}_{n_k} \rightarrow \mathcal{A}_\infty$ the unique unital $^*$-homomorphism satisfying $\pi_{\infty,k+1} \circ \pi_{k+1,k} = \pi_{\infty,k}$.
A \emph{maximal abelian self-adjoint algebra (masa)} $\cl{D}$ of a C$^*$-algebra $\cl{A}$ is a C$^*$-subalgebra of $\cl{A}$  that is abelian and not properly contained in any other abelian C$^*$-subalgebra of $\cl{A}$. 
In the context of UHF algebras it follows by \cite[Proposition 4.1]{Pow93} that $\cl{D}_\infty:=\displaystyle\lim_{\longrightarrow}(\cl{D}_{n_k},\pi_{k+1,k})$ is a masa, called the \emph{canonical masa}, in $\cl{A}_\infty:=\displaystyle\lim_{\longrightarrow}(\mathcal{M}_{n_k},\pi_{k+1,k})$, where $\mathcal{D}_n$ denotes the C$^*$-algebra of diagonal $n \times n$ matrices.

\par We highlight the fact that UHF algebras arise as full C$^*$-algebras of \'etale groupoids, slightly generalizing \cite[Example 2.3.5]{Sim18}.
A family $\{X_n\}_{n\in\bb{N}}$ of finite sets is called \emph{inductive} if $X_{n+1}=X_n\times[d_n]$ for every $n\in\bb{N}$, where $(d_n)_{n\in\bb{N}}$ are positive integers and $[d_n]=\{1,\ldots, d_n\}$.
Next, consider $X=\prod_{n=1}^\infty[d_n]$ to be one-side infinite strings where the $n$-th letter takes values from $1$ up to $d_n$, for every $n\in\bb{N}$. 
A \emph{word} is a string $x =x_k x_{k+1}\cdots x_{l}$, where $x_i\in [d_i]$, $i = k,\dots,l$, $k <l$; 
we write $|x| := l-k$ for its length;
if $k=1$ we call the word $x$ \emph{initial}, so that $X_n$ is the set  of initial words of length $n$.
We will use the notation $xx'$ for the concatenation of 
$x\in \prod_{i=k}^l [d_i]$, for some $k < l$,  
and  $x'\in \prod_{i=l+1}^m [d_i]$, for some $l < m$, or $x'\in \prod_{i=l+1}^{\infty} [d_i]$. 

Let $\mathcal{R}_X$ be the equivalence relation on $X$ given by $(x,y)\in \mathcal{R}_X$ if and only if there exists $n\in\bb{N}$ such that $x_i=y_i$, for all $i\geq n$. 
Note that $\cl{R}_X=\{(x,y)\mid x\sim y\}$ is a groupoid, where the unit space is $\mathcal{R}_X^{(0)} = \{(x,x) \mid x\in X\}$, the range and source maps are given by $r(x,y)=x$, $s(x,y)=y$,  where we customary identify $\mathcal{R}_X^{(0)}$ with $X$, the composition by $(x,y)(y,z)=(x,z)$, and the inverse by $(x,y)^{-1}=(y,x)$, for all $x,y,z\in X$.
The \emph{cylinder sets}
\begin{align} 
Z^n(u,v):=\left\{(ux,vx) \mid x\in 
\prod_{i= n+1}^\infty[d_i]
\right\},\; u,v\in 
X_n, 
\end{align}
form a base for a second countable locally compact Hausdorff topology on $\mathcal{R}_X$ turning it into a topological groupoid.
In particular, we have that any open subset $\Omega\subset\cl R_X$ contains $Z^{n+k}(ux,vx)$ if $(uy,vy)\in \Omega$, for some $y=xx'\in \prod_{i=n+1}^\infty [d_i]$, $x \in X_k$.

For convenience, we collect the following facts.
Let $n \leq m$ be positive integers, and let $u,v \in X_n$ and $x,y \in X_m$ be initial words.
Then we have
\begin{align}\label{eqn:intersection-cylinder-sets}
    Z^n(u,v) \cap Z^m(x,y) 
    = 
    \begin{cases}
        Z^m(x,y), \text{ if } \exists z, |z|=m-n : uz=x \text{ and } vz=y \\
        \emptyset, \text{ otherwise. }
    \end{cases}
\end{align}
Moreover, we have
\begin{align}\label{eqn:disjoint-union-cylinder-sets}
    Z^n(u,v) = \bigsqcup_{z \in [d_{n+1}]} Z^{n+1}(uz,vz) 
    = \bigsqcup_{|z'| = k} Z^{n+k}(uz',vz'),
\end{align}
for all $k \in \mathbb{N}$.

By \cite[Example 3.2.6]{Sim18} the groupoid C$^*$-algebra $\mathrm{C}^*(\cl R_X)$ is *-isomorphic to the above described UHF algebra $\cl{A}_\infty=\displaystyle\lim_{\longrightarrow}(\mathcal{M}_{X_k},\pi_{k+1,k})$.
For $u,v \in X_n$, we denote by $e_{u,v}^n$ the image under $\pi_{\infty,n} : \mathcal{M}_{X_n} \rightarrow \mathcal{A}_\infty$ of the according matrix unit $E_{u,v} \in \mathcal{M}_{X_n}$.
Under the isomorphism $\mathcal{A}_\infty \cong \mathrm{C}^*(\mathcal{R}_X)$, the element $e_{u,v}^n$ is identified with the indicator function $1_{Z^n(u,v)} \in \mathrm{C}_\mathrm{c}(\mathcal{R}_X)$, and we will freely switch between these two notations without further mentioning.
We set $$\mathcal{A}_n := \mathrm{span}\{e^n_{u,v} \mid u,v\in X_n \} = \pi_{\infty,n}(\mathcal{M}_{X_n}) \subseteq \mathcal{A}_\infty,$$ and note that $\overline{\bigcup_{n=1}^\infty \cl A_n} = \cl A_{\infty}$.

For $n \leq m$, (\ref{eqn:disjoint-union-cylinder-sets}) is equivalent to
\begin{align}\label{eqn:image-matrix-units}
    e^n_{u,v} = \sum_{z \in [d_{n+1}]} e^{n+1}_{uz,vz}
    = \sum_{|z'|=k} e^{n+k}_{uz',vz'},
\end{align}
for all $k\in\mathbb{N}$.
This implies, for $|u|=|v|=n < m = |x|=|y|$, 
\begin{align}\label{eqn:product-matrix-units}
	\begin{split}
        e^n_{u,v} e^m_{x,y} 
		= \sum_{|z|=m-n} e_{uz,vz}^m e_{x,y}^m 
		= \sum_{|z|=m-n} \delta_{vz,x} e_{uz,y}^m 
		= \delta_{v,x_1 \dots x_n} e_{ux_{n+1}\cdots x_m,y}^m.
	\end{split}
\end{align}

Let $\Omega\subseteq \cl R_{X}$ be a positivity domain, i.e.\ an open subset which is symmetric and contains the diagonal $\mathcal{R}_X^{(0)} \cong X$. 
By \autoref{groupoid_op_sys}, the operator system $\cl E(\Omega)$ is hyperrigid and the generated C$^*$-subalgebra of $\mathrm{C}^*(\cl R_{X})\cong \cl A_{\infty}$ coincides with its C$^*$-envelope, to wit $\mathrm{C}^*(\cl E(\Omega)) \cong \mathrm{C}^*_\mathrm{e}(\mathcal{E}(\Omega))$. 
We now characterise which operator subsystems of $\mathcal{A}_\infty$ arise in this way.
\autoref{prop:Masa-bimodule-positivity-domain} is essentially \cite[Theorem 6.12]{MT18}.
We give a proof that highlights the connection with our notion of positivity domains for groupoids and points towards a generalization to $\mathrm{C}_\mathrm{c}(\mathcal{G}^{(0)})$-bimodular operator subsystems of AF algebras. 

In the proof, we apply the following results of Power's \cite{Pow93}.
Denote by $\mathcal{A}_n := \pi_{\infty,n}(\mathcal{M}_{X_n}) \subseteq \mathcal{A}_\infty$ the image of the $|X_n| \times |X_n|$ matrix algebra in the UHF algebra $\mathcal{A}_\infty$.
Let $\widetilde{\cl A}_n := \overline{\mathcal{D}_\infty \cdot \mathcal{A}_n \cdot \mathcal{D}_\infty}$ be the $\cl D_\infty$-bimodule generated by $\cl A_n$. 
By \cite[Proposition 4.6]{Pow93}, there is a $\cl D_\infty$-bimodule conditional expectation $P_n:\cl A_\infty\to\widetilde{\cl A}_n$. Note that
$\widetilde{\cl A}_n=\text{span}\{de_{u,v}^n \mid 
u,v\in X_n,
d\in \cl D_\infty \}$.
We will need the following  lemma.

%%%%%%%%%%%%%%%%%%%%%%%%%%%
\begin{lemma}\label{lemm_supp}
Let $n < m $ be positive integers, and let  
$u,v\in X_m$ initial words of length $m$.
Then $P_n(e_{u,v}^m)=e_{u,v}^m$ if $Z^m(u,v)\subseteq Z^n(y,z)$ for some 
initial words $y,z \in X_n$, 
and $P_n(e_{u,v}^m)=0$ otherwise.  
\end{lemma}

\begin{proof}
 If $Z^m(u,v) \subseteq Z^n(y,z)$, for some 
 initial words $y, z \in X_n$ of length $n$, 
 then $P_n(e_{u,v}^m)=e_{u,v}^m$ is clear, since $P_n|_{\widetilde{\mathcal{A}}_n} = \mathrm{id}_{\widetilde{\mathcal{A}}_n}$ and $e_{u,v}^m \in \widetilde{\mathcal{A}}_n$.
 
 If $Z^m(u,v)$ is not a subset of $Z^n(y,z)$, for any 
 initial words $y, z \in X_n$ of length $n$, 
 then by (\ref{eqn:intersection-cylinder-sets}) and (\ref{eqn:disjoint-union-cylinder-sets}), we have $Z^m(u,v)\cap Z^m(yx,zx)=\emptyset$ for all words $x$ with $|x|=m-n$, and hence $e_{yx,yx}^me_{u,v}^me_{zx,zx}^m=0$, for all such $x$. 
 Note that $$P_n(e_{u,v}^m)=\sum_{y,z \in X_n}d_{y,z}e_{y,z}^n= \sum_{y,z \in X_n, |x|=m-n}d_{y,z}e_{yx,zx}^m,$$ for some  $d_{y,z}\in \cl D_\infty$.
 Applying the fact that $P_n$ is a $\mathcal{D}_\infty$-bimodule map, we have 
\begin{align}
    P_n(e_{u,v}^m)
    &=\sum_{y,z \in X_n,  |x|=m-n}e_{yx,yx}^mP_n(e_{u,v}^m)e_{zx,zx}^m\\
    &=P_n\Big(\sum_{y,z \in X_n, |x|=m-n}e_{yx,yx}^me_{u,v}^me_{zx,zx}^m\Big)=0,
\end{align}
and the proof is complete.
\end{proof}

%%%%%%%%%%%%%%%%%%%%%%%%%%%
\begin{theorem}\label{prop:Masa-bimodule-positivity-domain}
Let $X=\prod_{i=1}^\infty[d_i]$. 
Denote by $\mathcal{A}_\infty := \underset{\longrightarrow}{\lim} (\mathcal{M}_{X_n}, \pi_{n+1,n})$ the associated UHF algebra, and let $\mathcal{D}_\infty \subseteq \mathcal{A}_\infty$ be the canonical masa. 
Then for every concrete closed operator $\cl{D}_\infty$-subsystem $\cl{S}_\infty\subseteq\cl{A}_\infty$, there exists a positivity domain $\Omega \subseteq \mathcal{R}_X$, such that 
\begin{align} 
\mathcal{S}_\infty \cong \overline{\mathcal{E}(\Omega)},
\end{align}
under the canonical identification $\mathcal{A}_\infty \cong \mathrm{C}^*(\mathcal{R}_X)$. 
\end{theorem}

\begin{proof}
We begin by giving a candidate for the positivity domain $\Omega \subseteq \mathcal{R}_X$.
To this end, consider $\cl S_n :=\cl S_\infty\cap \cl A_n$ and observe that $\cl S_n\subseteq\cl S_{n+1}$. We obtain an inductive sequence of concrete operator $\cl D_n$-systems, where $\cl D_n := \cl D_\infty\cap \cl A_n$ is the image $\pi_{\infty,n}(\cl D_{X_n})$ of the diagonal matrix algebra $\mathcal{D}_{X_n} \subseteq \mathcal{M}_{X_n}$ in $\mathcal{A}_\infty$, for every $n \in \mathbb{N}$. 
By \cite[Theorem 4.7]{Pow93}, \cite[Proposition 6.7]{MT18}, we have 
\begin{align} 
\cl S_\infty=\overline{\bigcup_{n\in\mathbb N}\cl S_n}.
\end{align}
We write $\mathcal{S}_{X_n} := \pi_{\infty,n}^{-1}(\mathcal{S}_n) \subseteq \mathcal{M}_{X_n}$, and note that $\mathcal{S}_{X_n}$ is $\cl D_{X_n}$-bimodular.
Hence, the operator system $\cl S_{X_n}$ is a graph operator system for a graph $G_n$ on the vertex set $X_n = \prod_{i=1}^n [d_i]$ consisting of initial words of length $n$ \cite{PPS89}:
two vertices $u, v \in X_n$ of $G_n$ are adjacent or equal, denoted by $u\simeq_{G_n} v$, if and only if the matrix unit $E_{u,v} \in \mathcal{M}_{X_n}$ is an element of $\cl S_{X_n}$, if and only if $e^n_{u,v} \in \mathcal{S}_n \subseteq \mathcal{A}_\infty$.
As $\mathcal{S}_n \subseteq \mathcal{S}_{n+1}$, and hence $\pi_{n+1,n}(\cl S_{X_n})\subseteq \cl S_{X_{n+1}}$, we note that $u\simeq_{G_n}v$ implies $uw\simeq_{G_{n+1}}vw$, for any $w\in [d_{n+1}]$. 

Let 
\begin{align} 
\Omega_n:=\bigcup_{u\simeq_{G_n}v} Z^n(u,v)\subset \cl R_X, \quad \Omega :=\bigcup_{n\in\mathbb N}\Omega_n.
\end{align}
Clearly, we have $\Omega_{n} \subseteq \Omega_{n+1}$ and $\Omega_n$ is an open and symmetric subset of $\mathcal{R}_X$ which contains the diagonal, i.e.\ $\Omega_n$, for every $n\in\mathbb{N}$, and thus also $\Omega$ is a positivity domain.
We now show that $\overline{\cl E(\Omega)} \cong \cl S_\infty$ under the identification $\mathrm{C}^*(\mathcal{R}_X) \cong \mathcal{A}_\infty$, which will conclude the proof of this proposition. 

We first note that the inclusion $\cl S_n \subseteq  \overline{\cl{E}(\Omega)}\cap \cl A_n$ is immediate by construction:
indeed, assume $e_{u,v}^n \in \mathcal{S}_n$, equivalently $E_{u,v} \in \mathcal{S}_{X_n}$, equivalently $u \simeq_{G_n} v$.
By construction, it follows that $Z^n(u,v) \subseteq \Omega_n \subseteq \Omega$, showing that $e_{u,v}^n = 1_{Z^n(u,v)} \in \mathcal{E}(\Omega) \cap \mathcal{A}_n$.

To see the other inclusion $\cl S_n \supseteq \overline{\cl{E}(\Omega)}\cap \cl A_n$ we follow similar arguments as in the proof of \cite[Theorem 6.12]{MT18}. 
Since $\mathcal{A}_n = \mathrm{span}\{e_{u,v}^n \mid u,v \in X_n\}$, and $\mathcal{E}(\Omega)$ is $\cl D_n$-bimodular, it is enough to see that any $e_{u,v}^n\in  \overline{\mathcal{E}(\Omega)}\cap \cl A_n$ belongs to $\cl S_{n}$, for all $n \in \mathbb{N}$,  $u,v\in X_n$. 

Fix a positive integer $n\in\mathbb{N}$ and words $u,v \in X_n$ such that $e_{u,v}^n \in \overline{\mathcal{E}(\Omega)} \cap \mathcal{A}_n$.
Let $(f_m)_{m\in\mathbb{N}} \in \cl E(\Omega)$ be a sequence of continuous functions with $\mathrm{supp}(f_m) \subseteq \Omega$, and such that $1_{Z^n(u,v)} = \lim_{m\to\infty} f_m$ in $\mathrm{C}^*(\cl R_X)$.
Note that $\mathrm{span}\{f \in \mathrm{C}_\mathrm{c}(\mathcal{R}_X) \mid \exists n \in \mathbb{N}, u,v \in X_n : \supp(f) \subseteq \Omega \cap Z^n(u,v) \}$ is dense in $\mathcal{E}(\Omega)$, and recall from the proof of \autoref{main_thm} that $\mathcal{E}(\Omega)$ is spanned by continuous compactly supported functions whose support is a bisection contained in $\Omega$.
Recall from the paragraph preceding \autoref{lemm_supp}, the $\cl D_\infty$-bimodule conditional expectation $P_n : \mathcal{A}_\infty \rightarrow \widetilde{\mathcal{A}}_n$, where $\widetilde{\mathcal{A}}_n$ is the $\cl D_\infty$-bimodule generated by $\cl A_n$.
Since $e_{u,v}^n \in \mathcal{A}_n$ is preserved by $P_n$ and by continuity of $P_n$, we have
\begin{align}\label{eq_cond_exp}
e_{u,v}^n = \lim_{m\to\infty} P_n(f_m),
\end{align}
where we are now viewing $f_m$ as an element of $\mathcal{A}_\infty$ (and will be doing so in the sequel).

For any $f \in \mathrm{C}_\mathrm{c}(\mathcal{R}_X)$, $l \in \mathbb{N}$ and $s,t \in X_l$, observe that if $\supp(f) \subseteq Z^l(s,t)$, then there is a function $d_{s,t}\in \mathrm{C}(X)$ such that $f(x,y)=d_{s,t}(x)1_{Z^l(s,t)}(x,y)$: 
we may let $d_{s,t}(su)=f(su,tu)$, for $u\in \prod_{k\geq l} [d_k]$, and $d_{s,t}(x)=0$ otherwise. 
For every index $m$, the support of the function $f_m$ can be covered by finitely many bisections $Z^l(s,t)$ contained in $\Omega$.
Using a partition of unity, we write $f_m$ as a linear combination of $f_m^{l;s,t}$ each with support in $Z^l(s,t)$.
By \autoref{lemm_supp}, for every index $m$, we have $P_n(f_m^{l;s,t}) = 0$ if $Z^l(s,t) \nsubseteq \bigcup_{ p,q \in X_n} Z^n(p,q)$; 
hence, we may assume without loss of generality that $\mathrm{supp}(f_m) \subseteq \Omega \cap \bigcup_{ p,q \in X_n} Z^n(p,q)$.
Furthermore, as 
$$e_{u,v}^n = \lim_{m\to\infty} e_{u,u}^nP_n(f_m)e_{v,v}^n=\lim_{m\to\infty} P_n(e_{u,u}^n f_me_{v,v}^n),$$ and $\supp (e_{u,u}^n f_me_{v,v}^n)\subseteq Z^n(u,v)$, we can hence assume that $f_m$ is a sum of functions $f_m^{l;s,t}$ each supported in $Z^l(s,t)\subseteq \Omega\cap Z^n(u,v)$.

Next we will prove for our fixed $n \in \mathbb{N}$ and $u,v\in X_n$ that $Z^n(u,v)\subseteq \Omega$, which by \cite[Proposition 7.3]{Pow93} would imply that $e_{u,v}^n\in \cl S_n$ and hence the inclusion $\overline{\cl E(\Omega)}\cap\cl A_n\subseteq \cl S_n$.  
Consider the set
\begin{align} 
Y:=\bigcup_{m\in\mathbb{N}}\supp (P_n(f_m))\subseteq \Omega\cap Z^n(u,v)\subseteq Z^n(u,v).
\end{align}
We observe first that $\overline{Y}=\overline{\Omega\cap Z^n(u,v)}= Z^n(u,v)$. 
Indeed, assuming that $\overline{Y}\subsetneq Z^n(u,v)$, and hence $\overline{r(Y)}\ne Z^n(u,u)$, as $r$ is a homeomorphism on bisections, we can find a positive element $a\in \cl D_\infty^+$ in the masa with $a=1$ on $\overline{r(Y)}$ and $a\ne 1_{Z^n(u,u)}$. 
As $P_n(f_m)$ is supported in $\overline{Y}$, and hence  $aP_n(f_m)=P_n(f_m)$, for all $m\in\mathbb{N}$, we get from the approximation (\ref{eq_cond_exp}) that $a1_{Z^n(u,v)}=1_{Z^n(u,v)}$, giving a contradiction;
thus
\begin{align}\label{eqn:closure-Y}
    \overline{Y} = Z^n(u,v).
\end{align}

Now choose a set of disjoint bisections $\{Z^{n_k}(s_k,t_k)\}_{k\in K}$ in $\Omega\cap Z^n(u,v)$, where $K$ is finite or countable, such that  $\supp (P_n(f_m))\subseteq \bigcup_{k\in F_m} Z^{n_k}(s_k,t_k)$ for some finite subset $F_m\subseteq K$.  
By (\ref{eqn:closure-Y}) we have $\overline{Y}\subseteq \overline{\bigcup_{k\in K} Z^{n_k}(s_k,t_k)}\subseteq Z^n(u,v)=\overline{Y}$, giving 
\begin{align} 
Z^n(u,v)=\overline{\bigcup_{k\in K} Z^{n_k}(s_k,t_k)}.
\end{align}
 
If $K$ is finite, we obtain $Z^n(u,v)=\bigcup_{k\in K} Z^{n_k}(s_k,t_k)\subseteq \Omega$, and hence the statement.
Assume that $K$ is infinite. 
Then from (\ref{eq_cond_exp}), $1_{Z^n(u,v)}$ is the $\mathrm{C}^*(\cl R_X)$-limit of sums of  functions $g_{m,k}\in \mathrm{C}_\mathrm{c}(\cl R_X)$ (as $m \rightarrow \infty$) supported in $Z^{n_k}(s_k,t_k)$, $k\in F_m$. 
We have 
\begin{align} 
g_{m,k}(\gamma)
=d_{s_k,t_k}^m(r(\gamma))1_{Z^{n_k}(s_k,t_k)}(\gamma)
=d_{s_k,t_k}^m(r(\gamma))1_{Z^{n_k}(s_k,s_k)}(r(\gamma))1_{Z^n(u,v)}(\gamma), 
\end{align}
for some  $d_{s_k,t_k}^m\in \mathrm{C}(X)$ and all $\gamma\in\cl R_X$. 
It follows that $1_{Z^n(u,u)}(r(\gamma))$ is the uniform and hence pointwise limit of sums of $d_{s_k,t_k}^m(r(\gamma))1_{Z^{n_k}(s_k,s_k)}(r(\gamma))$:
\begin{align}
    & \|1_{Z^n(u,u)}-\sum_{k\in F_m} d_{s_k,t_k}^m1_{Z^{n_k}(s_k,s_k)}\|_{\mathrm{C}(X)}^2
    = \\
    & \hspace{2cm} = \Big\|\Big(1_{Z^n(u,u)}-\sum_{k\in F_m} d_{s_k,t_k}^m1_{Z^{n_k}(s_k,s_k)}\Big)\ast 1_{Z^n(u,v)} \ast \\
    & \hspace{6cm} \ast 1_{Z^n(u,v)}^* \ast \Big(1_{Z^n(u,u)}-\sum_{k\in F_m} d_{s_k,t_k}^m1_{Z^{n_k}(s_k,s_k)}\Big)^*\Big\|_{\mathrm{C}^*(\cl R_X)} \\
    & \hspace{2cm} = \Big\| \Big(1_{Z^n(u,u)}-\sum_{k\in F_m} d_{s_k,t_k}^m1_{Z^{n_k}(s_k,s_k)}\Big)\ast 1_{Z^n(u,v)} \Big\|_{\mathrm{C}^*(\cl R_X)}^2
    \to 0,
\end{align}
by 
the C$^*$-identity, using that $1_{Z^n(u,u)}=1_{Z^n(u,v)}\ast 1_{Z^n(v,u)}$ and that the function in the first line is
supported on $Z^n(u,u)$.

Assuming that $Z^n(u,v)\ne \bigcup_{k\in K} Z^{n_k}(s_k,t_k)$ we find $\gamma\in Z^n(u,v)\setminus \bigcup_{k\in K} Z^{n_k}(s_k,t_k)$ and hence $r(\gamma)\in Z^n(u,u)\setminus \bigcup_{k\in K} Z^{n_k}(s_k,s_k)$, giving $1_{Z^{n_k}(s_k,s_k)}(r(\gamma))=0$, for all $k\in K$, while $1_{Z^n(u,u)}(r(\gamma))=1$. 
A contradiction, that shows that $Z^n(u,v)=\bigcup_{k\in K} Z^{n_k}(s_k,t_k)\subset\Omega$. 
We conclude that $\overline{\mathcal{E}(\Omega)} \cap \mathcal{A}_n \subseteq \mathcal{S}_n$ which finishes the proof.
\end{proof}

%%%%%%%%%%%%%%%%%%%%%%%%%%%
\begin{corollary}\label{masa-bim-uhf}
Let $\cl{A}_\infty$ be a UHF algebra with canonical masa $\cl{D}_\infty$. 
Then every concrete operator $\cl{D}_\infty$-system $\cl{S}_\infty\subseteq\cl{A}_\infty$ is hyperrigid in $\mathrm{C}^*(\cl{S}_\infty)$. 
Furthermore, if the graph corresponding to the graph operator system $\cl{S}_\infty\cap\cl{A}_n$ is connected, for every $n\in\bb{N}$, then $\mathrm{C}^*(\cl{S}_\infty)=\cl{A}_\infty.$ 
\end{corollary}

\begin{proof}
By \autoref{prop:Masa-bimodule-positivity-domain}, we have $\mathcal{S}_\infty = \overline{\mathcal{E}(\Omega)}$, and by \autoref{groupoid_op_sys} this operator system is the closure of a hyperrigid operator system, hence hyperrigid.

The second claim follows once we show that the positivity domain $\Omega$ constructed in the proof of \autoref{prop:Masa-bimodule-positivity-domain} is non-degenerate, whenever all the graphs $G_n$ associated to the graph operator systems $\mathcal{S}_\infty \cap \mathcal{A}_n$ are connected.
To this end, note that the graph $G_n$ is connected, if and only if $\mathrm{C}^*(\mathcal{S}_{X_n}) = \mathcal{M}_{X_n}$, if and only if $\mathrm{C}^*(\mathcal{S}_n) = \mathcal{A}_n \subseteq \mathcal{A}_\infty$, if and only if $Z^n(u,v) \subseteq \mathcal{R}_X(\Omega_n)$ for all words $u,v \in X_n$.
Recall that $\mathcal{R}_X(\Omega_n)$ here denotes the subgroupoid of $\mathcal{R}_X$ generated by the positivity domain $\Omega_n$.
We obtain the following chain of inclusions
\begin{align}
     \mathcal{R}_X(\Omega)
      = \mathcal{R}_X\Big(\bigcup_{n\in\mathbb{N}} \Omega_n\Big)
       \supseteq \bigcup_{n\in\mathbb{N}} \mathcal{R}_X(\Omega_n)
     \supseteq \bigcup_{n\in\mathbb{N}} \bigcup_{u,v\in X_n} Z^n(u,v)
    = \mathcal{R}_X,
\end{align}
showing that the positivity domain $\Omega \subseteq \mathcal{R}_X$ is non-degenerate.
\end{proof}

%%%%%%%%%%%%%%%%%%%%%%%%%%%
\begin{remark}
\autoref{prop:Masa-bimodule-positivity-domain} can be generalised to AF C$^*$-algebras in the following way.
Let $\cl B_\infty$ be a unital AF C$^*$-algebra with finite-dimensional subalgebra chain
\begin{align} 
\cl B_1\subseteq \cl B_2\subseteq\ldots\subseteq \cl B_n\subseteq\ldots \subseteq \cl B_\infty,
\end{align}
and let $\cl C_\infty$ be the associated canonical masa with the chain
\begin{align} 
\cl C_1\subseteq \cl C_2\subseteq\ldots\subseteq \cl C_n\subseteq\ldots \subseteq \cl C_\infty.
\end{align}
Let $\cl C_\infty\subseteq \cl S\subseteq \cl B_\infty$ be a closed $\cl C_\infty$-bimodule operator system.
By \cite[Theorem 4.7]{Pow93}, $\cl S$ is the closed union of the chain
\begin{align} \cl S_1\subseteq \cl S_2\subseteq\ldots\subseteq \cl S_n\subseteq\ldots,
\end{align}
where $\cl S_n=\cl S\cap \cl B_n$.

Let $\{e_{i,j}^k\}$ be the matrix unit system for $\cl B_\infty$ associated to the chains $\{\cl C_k\}$ and $\{\cl B_k\}$. 
Let $\widehat{\cl C}_\infty$ be the Gelfand spectrum of $\cl C_\infty$. As in \cite[Chapters 5,7]{Pow93}, associate to each $e_{i,j}^k$  the subset $E_{i,j}^k\subseteq\widehat{\cl C}_\infty\times\widehat{\cl C}_\infty$ and define a topological equivalence relation 
\begin{align} 
\cl R(\cl B_\infty)=\bigcup E_{i,j}^k
\end{align}
and a topological binary relation of $\cl S$  
\begin{align} 
\Omega(\cl S) =\bigcup\{E_{i,j}^k \mid e_{i,j}^k\in \cl S_k\}\subseteq \cl R(\cl B_\infty).
\end{align}
In the case of the UHF algebra $\mathcal{A}_\infty = \underset{\longrightarrow}{\lim} (\mathcal{M}_{X_n}, \pi_n)$ from \autoref{prop:Masa-bimodule-positivity-domain}, we have $\widehat{\cl C}_\infty \simeq X=\prod_{i=1}^\infty[d_i]$, $\{E_{i,j}^k\}=\{Z(u,v)\}$ and $\Omega(\cl S)=\Omega$, see \cite[Example 5.2]{Pow93}.
As in the case of $\cl A_\infty$, the sets $E_{i,j}^k$ form a base of open sets for the topology on $\cl R(\cl B_\infty)$; 
moreover, either $E_{i,j}^k\cap E_{s,t}^l=\emptyset$ or $E_{s,t}^l\subseteq E_{i,j}^k$ if $l>k$, and for each $i,j,k$ there are matrix units $e_{i',j'}^{k+1}$ such that $e_{i,j}^k=\sum e_{i',j'}^{k+1}$, giving also the relation $E_{i,j}^k=\bigsqcup E_{i',j'}^{k+1}$. 
All results from \cite{Pow93} used in \autoref{prop:Masa-bimodule-positivity-domain} and \autoref{masa-bim-uhf} are applicable in the AF case and hence following exactly the steps in the corresponding proofs we obtain the same statements when replacing the UHF algebra $\cl A_\infty$ by a general unital AF algebra $\cl B_\infty$. 
\end{remark}

%%%%%%%%%%%%%%%%%%%%%%%%%%%
\section{Bimodularity, positivity domains, and spectral synthesis}
\label{sec:sp-synthesis} 
%%%%%%%%%%%%%%%%%%%%%%%%%%%

In this section we demonstrate that,  if the \'etaleness assumption on $\mathcal{G}$ is dropped, operator subsystems of the groupoid C$^*$-algebra which are $\mathrm{C}_\mathrm{c}(\mathcal{G}^{(0)})$-bimodular do not in general correspond to operator systems generated by positivity domains;
in fact, counterexamples appear already when $\mathcal{G}$ is the additive group $\mathbb{R}^n$ (viewed as a groupoid over the unit space $(\mathbb{R}^n)^{(0)} = \{0\}$).
Our first example arises from failure of \emph{reduced spectral synthesis}, and our second from failure of \emph{compact operator synthesis}.
We begin by recalling these notions.

Let $G$ be a locally compact group with $\mathrm{A}(G)$ the Fourier algebra of $G$, i.e.\ the Banach $^*$-algebra given by the matrix coefficients of the left regular representation $\lambda$ of $G$ \cite{Eym63}.
Let $S\subseteq G$ be a closed subset. 
If $G$ is abelian the Fourier algebra $\mathrm{A}(G)$ is isomorphic through the Fourier transform to the Banach algebra $\mathrm{L}^1(\widehat{G})$, where $\widehat{G}$ is the dual of $G$. 
For any ideal $J \subseteq \mathrm{A}(G)$, set 
\begin{align} 
\mathrm{null}(J) := \{s\in G \mid u(s)=0\text{ for all }u\in J\}.
\end{align}
To the set $S$ we associate the ideals  
\begin{align} 
\mathrm{I}(S) := \{u\in \mathrm{A}(G) \mid u(s)=0 \text{ for all } s\in S\},
\end{align}
and 
\begin{align} 
\mathrm{J}(S) := \overline{\{ u\in \mathrm{A}(G) \mid u \text{ has compact support disjoint from }  S\}}.
\end{align}
It is known that $\mathrm{J}(S)$ and $\mathrm{I}(S)$ are the minimal and maximal ideals respectively  with the null set being $S$ \cite{KL18}. 
The set $S$ is called \emph{a set of spectral synthesis} if $\mathrm{J}(S)=\mathrm{I}(S)$. 
    
The dual of $\mathrm{A}(G)$ is the von Neumann algebra $\mathrm{VN}(G)$ which is isomorphic to $\mathrm{L}^\infty(\widehat{G})$ in case $G$ is abelian. 
Write ${\mathrm C}_{\mathrm r}^*(G)$ and ${\mathrm C}^*(G)$ for the reduced and full C$^*$-algebras of $G$ and observe that $\overline{{\mathrm C}_\mathrm{r}^*(G)}^{\mathrm{w}^*}=\mathrm{VN}(G)$.
    
For $T\in \mathrm{VN}(G)$ one defines the support of $T$ as 
\begin{align} 
\supp(T) := \{t\in G \mid u\cdot T\ne 0 \text{ whenever } u\in \mathrm{A}(G), u(t)\ne 0\},
\end{align}
where 
\begin{align}\label{eqn:Fourier-action-von-Neumann}
\langle u\cdot T,v\rangle := \langle T, uv \rangle,
\text{ for } v\in \mathrm{A}(G).
\end{align}
Here $\langle\cdot,\cdot\rangle$ is the pairing between $\mathrm{VN}(G)$ and $\mathrm{A}(G)$ given by $\langle T,u\rangle=\langle T\xi,\eta\rangle$, for $u(\cdot)=\langle\lambda(\cdot)\xi,\eta\rangle \in \mathrm{A}(G)$.
We also define the support of a subspace $W\subseteq \mathrm{VN}(G)$ as the smallest closed subset $\supp(W) \subseteq G$ with $\supp(T) \subseteq \supp(W)$, for all $T \in W$.
It is known (see \cite{KL18}) that $$\mathrm{J}(S)^\perp=\{T\in \mathrm{VN}(G) \mid \supp(T) \subseteq S\}$$ and 
\begin{align} 
\mathrm{I}(S)^\perp=\overline{\{\lambda(\mu) \mid \mu\in \mathrm{M}(G), \supp\mu\subset S\}}^{\mathrm{w}^*}=\overline{\mathrm{span}\{\lambda(s) \mid s\in S\}}^{\mathrm{w}^*},
\end{align}
where $\mathrm{M}(G)$ denotes the set of all complex Radon measures on $G$ and 
\begin{align}
\langle\lambda(\mu)\xi,\eta\rangle_{\mathrm{L}^2}
= \int_G\langle\xi(t^{-1} \cdot ),\eta\rangle_{\mathrm{L}^2}\,\mathrm{d}\mu(t), \, \xi,\eta\in \mathrm{L}^2(G);
\end{align}
here, for all $u \in \mathrm{A}(G)$, $\mu\in \mathrm{M}(G)$ and $s\in G$, we have \[\langle\lambda(\mu),u\rangle=\int_Gu(t)\,\mathrm{d}\mu(t) \text{ and }\langle\lambda(s),u\rangle=u(s).\]
    
We write $\mathrm{PM}(S)$ for the set $\mathrm{J}(S)^\perp$ and call it the space of pseudomeasures supported in $S$. 
It is the largest  weak$^*$ closed $\mathrm{A}(G)$-invariant subspace of the group von Neumann algebra whose support is $S$,  where the $\mathrm{A}(G)$-action on $\mathrm{VN}(G)$ is defined in (\ref{eqn:Fourier-action-von-Neumann}). 
Correspondingly, $\mathrm{N}(S):=\mathrm{I}(S)^\perp$ is the smallest weak$^*$ closed $\mathrm{A}(G)$-invariant subset with support in $S$, meaning that any other weak$^*$ closed subspace $W$ which is $\mathrm{A}(G)$-invariant, with $\supp (W)=S$, satisfies $\mathrm{N}(S)\subseteq W\subseteq \mathrm{PM}(S)$. 
Clearly $S$ is a set of spectral synthesis if and only if $\mathrm{N}(S)=\mathrm{PM}(S)$.
We say that $S$ is a \emph{set of reduced spectral synthesis} if $\mathrm{N}(S)\cap \mathrm{C}^*_\mathrm{r}(G) = \mathrm{PM}(S)\cap \mathrm{C}^*_\mathrm{r} (G)$. 
We observe that for any open subset $\Omega\subseteq G$, we have $\lambda({\mathrm C}_\mathrm{c}(\Omega))\subseteq \mathrm{N}(\overline{\Omega})\cap {\mathrm C}_{\mathrm r}^*(G)$ and $\supp( \lambda({\mathrm C}_\mathrm{c}(\Omega)))=\overline{\Omega}$.

%%%%%%%%%%%%%%%%%%%%%%%%%%%
\begin{example}\label{spectral_synth}
By \cite[Theorem 3.13]{stt20}, every non-discrete locally compact group $G$ containing an open abelian subgroup admits a subset $S\subseteq G$ that is not a set of reduced spectral synthesis. 
An example is the sphere $\mathbb S^{n-1}\subseteq\mathbb R^n$  for $n\geq 4$, which is, in addition, symmetric; see \cite[Proposition 3.3]{stt20}. 
Moreover, by \cite[Theorems 3.13 and 3.16]{stt20}, every non-discrete locally compact abelian group $G$ contains a closed subset $S\subseteq G$ that fails reduced spectral synthesis and is the closure of its interior, 
$S=\overline{\operatorname{int}(S)}$.
Furthermore, if $G$ contains a symmetric subset that fails reduced spectral synthesis (such as $\mathbb S^{n-1}\subseteq\mathbb R^n$ for $n\geq 4$), then $S$ can be chosen to be symmetric.
By possibly adding the closure of an neighborhood $U_1 \ni e$ to $S$, with $U_2 \cap S = \emptyset$ for $U_2 \supset \overline{U_1}$ open, we may also assume that $S$ contains the unit of $G$. 
    
Let $G$ be abelian  and take $\mathcal E=\mathrm{PM}(S)\cap \mathrm{C}^*_\mathrm{r}(G)$ for a symmetric closed set $S\subseteq G$, which is the closure of its interior ${\rm int}(S)$, but  not a set of reduced spectral synthesis, so that ${\rm int}(S)$ is a positivity domain, naturally associated to $\cl E$.
Note that $\mathcal E$ is a selfadjoint weak$^*$ and hence norm closed subspace of $\mathrm{C}^*_\mathrm{r}(G)=\mathrm{C}^*(G)$: 
if $\supp (T)\subseteq S$ then $\supp (T^*) \subseteq S^{-1}$.  
Assuming that $\mathcal  E=\overline{\{ f\in \mathrm{C}_\mathrm{c}(G) \mid \supp (f)\subseteq \Omega\}}^{\mathrm{C}^*_\mathrm{r}(G)}$ for a positivity domain $\Omega$, we get that $\supp(\mathcal E)=\overline{\Omega}=S$ and $\mathcal E\subseteq \mathrm{N}(S)\cap \mathrm{C}^*_\mathrm{r}(G)\ne \mathrm{PM}(S)\cap \mathrm{C}^*_\mathrm{r}(G)$, a contradiction. 
\end{example}

In the previous example, the unit space is trivial, consisting only of the identity element of the group $G$. 
Consequently, $\mathrm{C}_\mathrm{c}(\cl G^{(0)})$-bimodularity imposes no restrictions, and there are no natural positivity domains that can be associated with an operator subsystem,
unless one imposes additional assumptions, such as the $\mathrm{A}(G)$-modularity considered in that example.

The next example deals with transformation groupoids $\cl G$ and closed $\rm C_{\rm c}(\cl G^{(0)})$-bimodular operator systems in $\mathrm{C}^*(\cl G)$. 
In this setting, one can naturally associate a positivity domain $\Omega\subseteq\cl G$ to each such operator system. 
However, like in the group case (with additional $A(G)$-modularity), the operator system is still not uniquely determined by its associated positivity domain; 
this phenomenon is related to the failure of compact operator synthesis, an operator analog of reduced spectral synthesis, see \cite{stt20}.
We begin by providing the setup for transformation groupoids.

Let $G$ be a locally compact group with a fixed Haar measure $m$ and the modular function $\Delta$. 
The group $G$ acts on itself by left translation generating a transformation groupoid $\mathcal G=G\times G$ with the unit space $\mathcal G^{(0)}=\{(e,x) \mid x\in G\}$, which we identify with $G$,  the source and the range maps given by $s(g,x)=g^{-1}x$, $r(g,x)=x$, the composition  $(g,x)\circ(h,g^{-1}x)=(gh, x)$ and the inverse $(g,x)^{-1}=(g^{-1}, g^{-1}x)$, $g,h,x\in G$. 
    
Consider $\mathrm{L}^2(G) := \mathrm{L}^2(G,m)$, the map $M: \mathrm{C}_0(G)\mapsto \cl B(\mathrm{L}^2(G))$, $f\mapsto M_f$, where $M_f$ is the multiplication operator by $f$, and the left regular representation $\lambda:G\to \cl B(\mathrm{L}^2(G))$. 
Then $(M,\lambda)$ is a covariant representation of the dynamical system $(\mathrm{C}_0(G),G, {\rm lt})$, i.e.
\begin{align} 
M(\alpha_x(f))=\lambda(x)M(f)\lambda(x)^*,
\end{align}
where $\alpha_x(f)(y)=f(x^{-1}y)$, $x,y\in G$, the induced left translation action on $\mathrm{C}_0(G)$.
By the Stone--von Neumann theorem \cite[Theorem 4.24]{Wil07}, $M\rtimes\lambda$ is a faithful irreducible representation of the full crossed product algebra $\mathrm{C}_0(G)\rtimes_{\rm lt}G\simeq \mathrm{C}^*(\mathcal{G})$ onto $\cl K(\mathrm{C}^2(G))$. 
Moreover, if $K\in \mathrm{C}_\mathrm{c}(G\times G)$ and 
\begin{align} 
f_K(s,t)=\Delta(s^{-1}t)K(t,s^{-1}t), s,t\in G,
\end{align}
then $f_K\in \mathrm{C}_\mathrm{c}(G\times G)$ and, by calculations in the proof of \cite[Theorem 4.24]{Wil07}, we get 
\begin{align} 
((M\rtimes\lambda)(f_K)\xi)(s)=\int_G K(s,t)\xi(t)dm(t), \xi\in \mathrm{L}^2(G), 
\end{align}
so that $(M\rtimes\lambda)(f_K)$ is a Hilbert-Schmidt integral operator on $L^2(G)$ with kernel $K\in \mathrm{C}_\mathrm{c}(G\times G)\subseteq \mathrm{L}^2(G\times G)$, which we will denote by $T_K$. 
Note that for $f\in \mathrm{C}_\mathrm{c}(G\times G)$, we have $(M\rtimes\lambda)(f)^*=(M\rtimes\lambda)(f^\dagger)$, where 
\begin{align} 
f^\dagger(x,y) = \Delta(x^{-1})\overline{f(x^{-1},x^{-1}y)} = \Delta(x^{-1})\overline{f((x,y)^{-1})} = \Delta(x^{-1}) f^*(x,y).
\end{align}
    
Let 
\begin{align}\label{eqn:psi}
\psi:G\times G\to G\times G, (x,y)\mapsto (y,x^{-1}y).
\end{align}
It is a homeomorphism, with
\begin{align}
\psi^{-1}(x,y) = (xy^{-1},x),
\end{align}
and $\supp (f_K)\subseteq\Omega$ if and only if $\supp (K)\subseteq\psi(\Omega)$ for any $\Omega\subseteq G\times G$.  
In particular, given an open subset $\Omega\subseteq G\times G$, we obtain the equality
\begin{align} 
(M\rtimes\lambda)(\cl E(\Omega))=\{T_K\in \cl K(\mathrm{L}^2(G))  \mid  K\in \mathrm{C}_\mathrm{c}(G\times G), \supp (K)\subseteq\psi(\Omega)\}.
\end{align}
As for $f,g\in \mathrm{C}_\mathrm{c}(G)$, the operator $M_fT_KM_g$ is a Hilbert-Schmidt integral operator with the kernel $(s,t) \mapsto f(s)K(s,t)g(t)$, we have $(M\rtimes\lambda)(\cl E(\Omega))$ and its closure is a bimodule over $\{M_f \mid f\in \mathrm{C}_\mathrm{c}(G)\}$.

Assume now $G$ is second countable so that $(G,m)$ is a standard measure space.  We recall now some definitions and facts from \cite{stt20}. 

The C$^*$-subalgebra $\{M_f \mid f\in \mathrm{L}^\infty(G)\}\subseteq \cl B(\mathrm{L}^2(G))$ is a masa (maximal abelian selfadjoint algebra) in $\cl B(\mathrm{L}^2(G))$.
A subspace $\cl W\subseteq\cl B(\mathrm{L}^2(G))$ is called a masa-bimodule if $M_bTM_a\in\cl W$ for all $T\in\cl W$ and $a,b\in \mathrm{L}^\infty(G)$. 
We say that a closed subset $\kappa\subseteq G\times G$ supports an operator $T\in\cl B(\mathrm{L}^2(G))$ if $M_{\chi_\beta}TM_{\chi_\alpha}=0$, for any measurable rectangle $\alpha\times\beta$ such that  $(\alpha\times\beta)\cap\kappa$ is marginally null, i.e. $(\alpha\times\beta)\cap\kappa\subseteq N\times G\cup G\times N$ with $m(N)=0$. 
Note that $T$ is supported in $\kappa$ if and only if $T^*$ is supported in $\sigma(\kappa)$, where $\sigma$ is the flip map $\sigma(x,y)=(y,x)$, $x,y\in G$. 
Define the support, $\supp(\cl W)$, of a subspace $\cl W\in\cl B(\mathrm{L}^2(G))$ as the smallest closed subset that supports each operator in $\cl W$ (\cite{arv_syn}).  
Write
\begin{align} 
\mathfrak M_{\rm max}(\kappa)=\{T\in \cl B(\mathrm{L}^2(G)) \mid T\text{ is supported by }\kappa\};
\end{align}
it is the largest weak$^*$ closed  masa-bimodule,  where each operator is supported by $\kappa$. 
There is also a smallest one, $\mathfrak M_{\rm min}(\kappa)$;
in particular all $T_K$, with $K\in \mathrm(G\times G)$ and $\supp (K)\subseteq \kappa$, belong to $\mathfrak M_{\rm min}(\kappa)$. 
We have 
\begin{align}
\supp (\mathfrak M_{\rm max}(\kappa))=\supp (\mathfrak M_{\rm min}(\kappa))=\kappa.
\end{align} 

The set $\kappa \subseteq G \times G$ is said to be of {\it operator synthesis} (with respect to the Haar measure) if
\begin{align} 
\mathfrak M_{\rm max}(\kappa)=\mathfrak M_{\rm min}(\kappa).
\end{align}
It is of {\it compact operator synthesis} if 
\begin{align} 
\mathfrak M_{\rm max}(\kappa)\cap\cl K(\mathrm{L}^2(G))=\mathfrak M_{\rm min}(\kappa)\cap\cl K(\mathrm{L}^2(G)).
\end{align}
For a closed subset $S\subseteq G$ set
\begin{align} 
S^*=\{(x,y)\in G\times G \mid yx^{-1}\in S\}.
\end{align}
We have the following  connection between (reduced) spectral synthesis and (compact) operator synthesis: $S$ is a set of (reduced) spectral synthesis if an only if $S^*$ is a set of (compact) operator synthesis (with respect to the Haar measure), see \cite[Theorem 4.3, 4.11]{lt06}, \cite[Theorem 5.1]{stt20}. 

%%%%%%%%%%%%%%%%%%%%%%%%%%%
\begin{example}\label{compact_synth}
Let $G$ be an abelian non-discrete compact group, and let $S\subseteq G$ be a symmetric closed subset which, as in the previous example, is not a set of reduced spectral synthesis. 
We may further assume that $S$ contains the identity element of $G$. 
It follows that $S^*\subseteq G\times G$ is not a set of compact operator synthesis. 
Set
\begin{align}
\cl E=\mathfrak M_{\rm max}(S^*)\cap \cl K(\mathrm{L}^2(G)).
\end{align}
 Since $T^* \in \mathfrak{M}_\mathrm{max}(S^*)$ is supported by $\sigma(S^*)$, if and only if $T$ is supported by $S^*$, and $\sigma(S^*) = (S^{-1})^* = S^*$, as $S$ is symmetric, the subspace $\mathcal{E} \subseteq \mathcal{K}(\mathrm{L}^2(G))$ is self-adjoint.

As in the previous example, suppose that, for some open subset $\Omega\subseteq G\times G$, we have
\begin{align} 
\cl E
=
\overline{\cl E(\Omega)}^{\,\cl K(\mathrm{L}^2(G))}
=
\overline{\{T_K \mid K\in \mathrm{C}_\mathrm{c}(G\times G),\ \supp (K)\subseteq\psi(\Omega)\}}^{\,\cl K(\mathrm{L}^2(G))},
\end{align}
where $\psi(x,y)=(y,x^{-1}y)$, for all $x,y\in G$.
Then
\begin{align}
\supp(\cl E)=\overline{\psi(\Omega)}=S^*,
\end{align}
and
\begin{align}
\cl E\subseteq
\mathfrak M_{\rm min}(S^*)\cap\cl K(\mathrm{L}^2(G)),
\end{align}
 which contradicts the assumption that $S$ is not of reduced spectral synthesis. 
\end{example}

However, if $G$ is discrete so that the groupoid
$\cl G=G\times G$ is étale, we have that any bimodular operator system is determined by its associated positivity domain.

%%%%%%%%%%%%%%%%%%%%%%%%%%%
\begin{proposition}
Let $G$ be discrete and let $\cl E \subseteq \mathrm{C}^*(\cl G)$ be a closed $\mathrm{C}_\mathrm{c}(\cl G^{(0)})$-bimodular self-adjoint operator subspace. 
Then 
\begin{align}
\cl E=\overline{\cl E(\Omega)}^{\,\cl K(\ell^2(G))},
\end{align}
where
\begin{align}
\Omega
= 
\psi^{-1}\left(
\left\{
(x,y)\in G\times G \mid
\text{there exists }T\in\cl E
\text{ such that }
\langle T\delta_y,\delta_x\rangle\neq 0
\right\}
\right),
\end{align}
 with $\psi : G \times G \to G \times G, (x,y) \mapsto (y,x^{-1}y)$, 
and where $\{\delta_x\}_{x\in G}$ denotes the standard orthonormal basis of $\ell^2(G)$.
\end{proposition}

\begin{proof}
Since $\cl E$ is self-adjoint, the set $\Omega$ is symmetric:  for $x,y\in G$ we have 
\begin{align} 
\psi((x,y)^{-1})=\psi(x^{-1}, x^{-1}y)=(x^{-1}y,y),
\end{align}
while $(x,y)^{-1}=\psi^{-1}(\sigma(\psi(x,y)))$, where $\sigma:G\times G\to G\times G$ is the flip map $(x,y)\mapsto(y,x)$. 
As $\psi(\Omega)$ is $\sigma$-invariant  by self-adjointness of $\mathcal{E}$,  we obtain that $\Omega^{-1}=\Omega$.

The $\mathrm{C}_\mathrm{c}(\cl G^{(0)})$-bimodularity of $\cl E$ implies that
\begin{align} 
T\in\cl E
\;\Longrightarrow\;
T^{E,F}:=M_{\chi_E}TM_{\chi_F}\in\cl E,
\end{align}
for every pair of finite subsets $E,F\subseteq G$. Since
$\cl E\subseteq\cl K(\ell^2(G))$, it follows that
\begin{align}
    \cl E
    =
    \overline{
        \left\{
            T^{E,F} \mid
            T\in\cl E,\;
            E,F\subseteq G \text{ finite}
        \right\}
    }^{\,\cl K(\ell^2(G))}.
\end{align}

For $T\in\cl E$ and finite subsets $E,F\subseteq G$, the operator $T^{E,F}$ is a finite-rank kernel operator whose kernel has finite support and is given by
\begin{align}
    K_{E,F}(s,t)
    =
    \langle T^{E,F}\delta_t,\delta_s\rangle
    =
    \langle TM_{\chi_F}\delta_t,
            M_{\chi_E}\delta_s\rangle,
    \quad s,t\in G.
\end{align}
Fix $T\in\cl E$ and finite subsets $E,F\subseteq G$. 
If $\langle T^{E,F}\delta_t,\delta_s\rangle\neq 0$, then, since $T^{E,F}\in\cl E$, we have $(s,t)\in\psi(\Omega)$. 
Hence $\operatorname{supp} K_{E,F}\subseteq\psi(\Omega)$, and it follows that
\begin{align}
\cl E
\subseteq
\overline{\cl E(\Omega)}^{\,\cl K(\ell^2(G))}.
\end{align}

For the reverse inclusion, suppose that $(x,y)\in\psi(\Omega)$. 
By the definition of $\Omega$, there exists $T\in\cl E$ such that $\langle T\delta_y,\delta_x\rangle\neq 0$.
Using the $\mathrm{C}_\mathrm{c}(\cl G^{(0)})$-bimodularity of $\cl E$, we obtain
\begin{align}
M_{\chi_{\{x\}}}TM_{\chi_{\{y\}}}
=
\langle T\delta_y,\delta_x\rangle\,
\delta_x\delta_y^*
\in\cl E.
\end{align}
Consequently, $\delta_x\delta_y^*\in\cl E$ for every $(x,y)\in\psi(\Omega)$.

Now let $S\in\cl E(\Omega)$, regarded as an operator in
$\cl K(\ell^2(G))$. 
Then
\begin{align}
S
=
\sum_{(x,y)\in\psi(\Omega)}
\langle S\delta_y,\delta_x\rangle\,
\delta_x\delta_y^*.
\end{align}
Since each rank-one operator $\delta_x\delta_y^*$ occurring in this sum belongs to $\cl E$, we conclude that $S\in\cl E$. 
Thus $\cl E(\Omega)\subseteq\cl E$, and since $\cl E$ is closed we get
\begin{align}
    \overline{\cl E(\Omega)}^{\,\cl K(\ell^2(G))}
    \subseteq\cl E,
\end{align}
as required.
\end{proof} 

%%%%%%%%%%%%%%%%%%%%
%%%%%%%%%%%%%%%%%%%%
%%%%%%%%%%%%%%%%%%%%

%\bibliographystyle{abbrv}
%\bibliography{biblio}

\end{document}